\documentclass{amsart}
\usepackage{amssymb}

\def\({\left(}
\def\){\right)}

\let\d\partial
\let\ph\varphi

\def\A{{\mathrm A}}
\def\D{{\mathrm D}}
\def\E{{\mathrm E}}
\def\C{{\mathbb C}}
\def\Z{{\mathbb Z}}
\def\R{{\mathbb R}}
\def\bx{{\boldsymbol x}}
\def\bt{{\boldsymbol t}}

\newtheorem{proposition}{Proposition}
\newtheorem{lemma}{Lemma}
\newtheorem{theorem}{Theorem}

\newtheorem{Definition}{Definition}

\begin{document}
\title[Unimodularity of the Poincar\'e polynomial]
{Unimodularity of Poincar\'e polynomials of Lie algebras for
semisimple singularities}
\author{Mamuka Jibladze}

\address{Department of Algebra, Razmadze Mathematical Institute, Tbilisi, Georgia}
\email{jib@rmi.acnet.ge}

\author{Dmitry Novikov}

\address{Department of Mathematics, Weizmann Institute of Science, Rehovot, Israel}
\email{dmitry.novikov@weizmann.ac.il}



\begin{abstract}
We single out a large class of semisimple singularities with the
property that all roots of the Poincar\'e polynomial of the Lie
algebra of derivations of the corresponding suitably (not
necessarily quasihomogeneously) graded moduli algebra lie on the
unit circle; for a still larger class there might occur exactly four
roots outside the unit circle. This is a corrected version of
Theorem 4.5 from \cite{ElKh}.
\end{abstract}

\maketitle

\section{Formulations}

\subsection{Tensor product of graded algebras and their gradings}

An important invariant of an isolated singularity $S$ given by a
holomorphic germ $f:(\C^n,0)\to(\C,0)$ is the algebra
$\C[[\bx]]\left/\(\frac{\d f}{\d\bx}\)\right.$. We will denote this
algebra by $A(S)$.

For complex isolated singularities $S_j$ given by germs $f_j\in
\C[[\bx_j]]$, $\bx_j\in\C^{n_j}$, $j=1,...,k$, their direct sum
$S=\bigoplus_{j=1}^kS_j$ is the singularity given by the germ of
$f(\bx)=\sum_{j=1}^kf_j(\bx_j)$ in $\C^{\sum_{j=1}^kn_j}$, see
\cite{AVG}. The algebra $A(S)$ of such a direct sum can be naturally
identified with the tensor product of algebras $A(S_j)$ of the
summands.

Simple singularities are those without moduli, see \cite{AVG} for
precise definition. They are classified by the Coxeter groups $A_k$,
$D_k$ and the exceptional ones $E_6$, $E_7$ and $E_8$, and are
denoted by $\A_k$, $\D_k$, $\E_6$, $\E_7$ and $\E_8$
correspondingly. A semisimple singularity is a direct sum of simple
singularities.

Suppose that each algebra $A(S_j)$ is endowed with a $\Z$-grading.
Then their tensor product $A(S)$ acquires a natural $\Z^k$-grading.

\begin{Definition}
For a finite-dimensional $\Z^k$-graded algebra $A=\oplus_{\alpha\in
\Z^k}A_\alpha$ we define its Poincar\'e polynomial $P_A$ by the
equality
$$
P_A(t_1,...,t_k)=\sum_\alpha\dim(A_\alpha)
t_1^{\alpha_1}...t_k^{\alpha_k}.
$$
\end{Definition}

The Poincar\'e polynomial $P(S)$ of $A(S)$ is equal to the product
of the Poincar\'e polynomials $P(S_j)(t_j)$ of $A(S_j)$:
$$
P(S)(\bt)=\prod_{j=1}^kP(S_j)(t_j),\qquad\bt=(t_1,...,t_k)\in \C^k.
$$

Lie algebras $L(S_j)$ of derivations of the graded algebras $A(S_j)$
inherit $\Z$-gradings following the convention deg$(\d/\d
x)=-$deg$(x)$ for all involved variables $x$.  Similarly, the
$\Z^k$-grading of the algebra $A(S)$ induces a $\Z^k$-grading on the
Lie algebra $L(S)$ of its derivations. Let $P_L(S_j)$ denote the
Poincar\'e polynomials of $L(S_j)$ corresponding to these gradings.
A theorem of Block \cite{Block} implies that the Poincar\'e
polynomial of $L(S)$ is
$$
P_L(S)(\bt)=\left[\sum_{j=1}^k
\frac{P_L(S_j)(t_j)}{P(S_j)(t_j)}\right]\prod_{j=1}^k P(S_j)(t_j).
$$

The above $\Z^k$-gradings can produce various $\Z$-gradings of
$A(S)$ and $L(S)$ via linear functionals $\phi:\Z^k\to\Z$. In other
words, one can define a $\Z$-grading on $A(S)$ and on $L(S)$ as a
linear combination with integer weights $w_j$ of the $\Z$-gradings
for $S_j$. The Poincar\'e polynomial $P_L^\phi(S)(t)$ of $L(S)$ with
respect to the resulting grading will be just
$P_L(S)(t^{w_1},...,t^{w_k})$.

For the semisimple singularities the most natural grading is the
quasihomogeneous one. Indeed, since the functions $f_j$ defining
simple singularities are quasihomogeneous, their sum $f=\sum
f_j(\bx_j)$ is quasihomogeneous as well.

That said, in this paper we follow a different choice of weights
considered in \cite{ElKh}. It corresponds to the linear functional
alluded to above given by $\phi(n_1,...,n_k)=n_1+...+n_k$, or, in
terms of the weights, $w_1=...=w_k=1$. This choice of weights,
together with a suitable choice of quasihomogeneous weights for
simple singularities, leads to our main result --- Theorem 1 below.

This theorem can be viewed as an alternative version of Theorem 4.5
from \cite{ElKh}. There it is stated that, in our terms, polynomials
$P_L^+(S)$ are \emph{unimodular}, i.~e. have all their roots on the
unit circle, for certain two classes of semisimple singularities,
i.~e., the sums of simple ones.

As it turns out, this is actually not true for the semisimple
singularities from the second named class. The first counterexample
occurs for the singularity $\D_{17}\oplus\E_7$, as shown in our
table below (note that, although some of our gradings are different,
those for $\D_{2k+1}$ and $\E_7$ are identical with those from
\cite{ElKh}).

The problem lies in the proof of 4.5 in \cite{ElKh}. That proof
involves an argument from the proof of Proposition 4.2 of the same
paper. Although the proposition itself is correct, its proof,
belonging to the first author of the present paper, contains an
erroneous claim, as pointed out by the second author (namely, it was
falsely assumed that for any two polynomials of the same degree with
positive leading coefficients which are both odd or both even and
all of their roots are real and lie in the segment $[-2,2]$, all
roots of their sum are also real and lie in the same segment).

Thus our aim is to describe certain class of semisimple
singularities for which the statement of Theorem 4.5 from
\cite{ElKh} holds true. This is done in our Theorem 1. As already
mentioned, we choose different quasihomogeneous gradings for some
simple singularities --- namely, for $\A_k$ and $\D_{2k}$. This
gives unimodularity of Poincar\'e polynomials for Lie algebras of a
class of semisimple singularities that is strictly larger than the
one corresponding to gradings considered in \cite{ElKh}.

\subsection{Two types of singularities and unimodularity of their Poincar\'e
polynomials}


The quasihomogeneous gradings of algebras $A(S)$ (and,
correspondingly, their Lie algebras $L(S)$) of simple singularities
$S=$ $\A_k$, $\D_m$ and $\E_j$ $(j=6,7,8)$ are defined by the
following quasihomogeneous weights $\omega$ of variables:
\begin{equation}\label{weights}
\begin{array}{lllll}
\A_k, &k\ge 1\quad & x^{k+1}& \omega(x)=2;\\
\D_m, &m\ge 4\quad & x^2y+y^{m-1}& \omega(x)=m-2,&\omega(y)=2;\\
\E_6&& x^3+y^4& \omega(x)=4,&\omega(y)=3;\\
\E_7&& x^3+xy^3& \omega(x)=3,&\omega(y)=2;\\
\E_8&& x^3+y^5& \omega(x)=5,&\omega(y)=3.
\end{array}
\end{equation}

\begin{proposition}[cf. \cite{ElKh}]\label{table}
Poincar\'e polynomials $P(S)$ and $P_L(S)$ for simple singularities
$S$ with respect to the above weights are:
$$
\begin{aligned}
P(\A_k)(t)&=\frac{1-t^{2k}}{1-t^2},&P_L(\A_k)(t)&=\frac{1-t^{2k-2}}{1-t^2};\\
P(\D_m)(t)&=\frac{(1+t^{m-2})(1-t^{m})}{1-t^2},&P_L(\D_m)(t)&=\frac{(1+t^{m-4})(1-t^{m})}{1-t^2};\\
P(\E_6)(t)&=\frac{(1+t^4)(1-t^9)}{1-t^3},&P_L(\E_6)(t)&=\frac{(1+t^4)(1-t^6)+1-t^9}{1-t^3};\\
P(\E_7)(t)&=\frac{(1+t^3)(1-t^7)}{1-t^2},&P_L(\E_7)(t)&=\frac{(1+t^3)(1+t)(1-t^4)}{1-t^2};\\
P(\E_8)(t)&=\frac{(1+t^5)(1-t^{12})}{1-t^3},&P_L(\E_8)(t)&=\frac{(1+t^5)(1-t^9)+1-t^{12}}{1-t^3}.
\end{aligned}
$$
\end{proposition}

These formul\ae\ conform with those from \cite{ElKh} except for
$\E_6$ and $\E_8$, as for the latter the non-quasihomogeneous
gradings are used in \cite{ElKh}. In these two remaining cases, the
polynomials are easily obtained by direct calculation.

We define semisimple singularities of type $\A\oplus\D$ to be the
direct sums of any number of $\A_k$'s and $\D_m$'s.

\begin{theorem}\label{main}
For the above choice of weights the Poincar\'e polynomial $P_L^+(S)$
of the Lie algebra of a semisimple singularity $S$ of type
$\A\oplus\D$ is unimodular, i.~e. has all roots on the unit circle
$\{|t|=1\}$.

For the same choice of weights the Poincar\'e polynomial
$P_L^+(S\oplus\E_7^{\oplus l})$ of the Lie algebra of the direct sum
of a semisimple singularity $S$ of type $\A\oplus\D$ and any number
$l$ of copies of $\E_7$ is either unimodular or has exactly  four
roots outside the unit circle $\{|t|=1\}$.
\end{theorem}

Note that $\E_6=\A_2\oplus\A_3$ and $\E_8=\A_2\oplus\A_4$. However,
the gradings of $\E_6$ and $\E_8$ in \eqref{weights} differ from
gradings of $\A_2\oplus\A_3$ and $\A_2\oplus\A_4$ in Theorem
\ref{main}. Indeed, for the direct sums of $\A_k$ the weights of all
variables in Theorem \ref{main} are chosen to be equal. For the
quiasihomogeneous weights the Poincar\'e polynomials $P_L(\E_6)$ and
$P_L(\E_8)$ are not palindromic, therefore not unimodular (see
Definitions \ref{def:pali}, \ref{def:uni} below).

\newpage

\section{Proofs}
\subsection{Palindromic polynomials}\label{ssec:generalities}

\begin{Definition}\label{def:pali}
A polynomial $\sum_{k=0}^na_kt^k$ of degree $n$ is called
\emph{palindromic} if $a_k=a_{n-k}$ for all $k$.
\end{Definition}

\begin{Definition}\label{def:uni}
A polynomial is called \emph{unimodular} if all its roots lie on the
unit circle.
\end{Definition}

A real unimodular polynomial $P(t)$ is either palindromic or becomes
palindromic after division by $1-t$.

Remarkably, Poincar\'e polynomials $P(S)$ and $P_L(S)$ for
singularities $S=$ $\A_k$, $\D_k$ and $\E_7$ are palindromic (this
is a consequence of the duality existing on the moduli algebras) and
moreover unimodular, see \cite{ElKh}.

Any product of two palindromic polynomials and a sum of palindromic
polynomials \emph{of the same degree} is again palindromic.
Therefore the Poincar\'e polynomials considered in Theorem
\ref{main} are palindromic.

\begin{lemma}
The function  $t^{-d}P(t)$, where $P$ is a palindromic polynomial of
degree $2d$ with real coefficients, takes real values on the unit
circle $\{|t|=1\}$.
\end{lemma}

\begin{proof}
Indeed, $t^{-d}P(t)=a_d+\sum_{k=1}^d a_{d-k}(t^k+t^{-k})$. Since
$t^k+t^{-k}$ is real on the unit circle and $a_k$ are real by
assumption, the result follows.
\end{proof}

Any palindromic polynomial $P(t)$ of odd degree is a product of
$1+t$ and a palindromic polynomial of even degree. Therefore  the
following is a corollary of the previous Lemma.

\begin{lemma}\label{lem:ratios of pali}
Let $Q=\sum\frac{P_j}{{\tilde P}_j}$ be a rational function, where
$P_j$, ${\tilde P}_j$ are palindromic polynomials with real
coefficients, and assume that the differences $\deg P_j-\deg{\tilde
P}_j$ are all equal to some number $2d$. Then $t^{-d}Q(t)$ takes
real values on the unit circle $\{|t|=1\}$.
\end{lemma}

For such rational functions $Q$ the function
$\ph(x)=e^{-2dix}Q(e^{2ix})$ is a real $\pi$-periodic function on
$\R$. Moreover, since $Q$ is real, the function $\ph(x)$ is even. In
fact, $\ph(x)$ is a rational function of the $\cos(kx)$, $k\in\Z$.

\begin{lemma}\label{lem:roots}
Let $\ph$ be as above and assume that $\ph$ has only simple poles on
some interval $[a,b]$. Denote by $n_+$ (resp. $n_-$) the number of
poles of $\ph$ on an interval $[a,b]$ with positive (resp. negative)
residue. Then the number of different zeros of $\ph$ on $[a,b]$ is
at least $|n_+-n_-|-1$.
\end{lemma}

\begin{proof}
Call an interval $(x_1,x_2)$ ``good'' if $x_1$, $x_2$ are two poles
of $\ph$ with residues of the same sign, and $\ph$ is continuous on
$(x_1,x_2)$. Evidently, $\ph$ has a zero in any good interval: $\ph$
tends to $+\infty$ at one endpoint, and to $-\infty$ at another.

The number of good intervals is at least $|n_+-n_-|-1$.
\end{proof}

In fact, one can prove a stronger result.

\begin{lemma}\label{lem:roots2}
Let $\ph$, $n_+$ and $n_-$ be as above and assume that $a,b$ are
neither zeros nor poles of $\ph$. Let $c$ be $1$ if $\ph(a)\ph(b)<0$
and $0$ otherwise.

Then the number of zeros of $\ph$ on $[a,b]$ counted with
multiplicities is at least $|n_+-n_-|-c$ and differs from the latter
expression by an even number.
\end{lemma}

\begin{proof}
Choose a smooth extension $\Phi$  of $\ph$ to $S=\R/(|a-b|+1)\Z\cong
S^1$ having exactly $c$ additional simple zeros and without
additional poles. It defines a smooth mapping $\Phi: S\to \R
P^1=\R\cup\{\infty\}$. Taking standard orientations and looking on
$\Phi^{-1}(\infty)$ we conclude that $\deg\Phi=n_--n_+$.

Zeros of $\ph$ fall into three classes: zeros of odd multiplicity
where $\ph$ increases (denote their number by $m_+$), zeros of odd
multiplicity where $\ph$ decreases (denote their number by $m_-$),
and zeros of even multiplicity. Evidently $\deg\tilde\Phi=m_+-m_-\pm
c$, and the claim easily follows.
\end{proof}

\subsection{Proof of Theorem \ref{main}}

The Poincar\'e polynomial $P_L(S)$ for $S$ of type $\A\oplus\D$ has
the following form:
$$
P(t)=\left[\prod_{j=1}^kP(\A_{k_j})(t)\prod_{j=1}^M
P(\D_{m_j})(t)\right]Q(t),
$$
where
$Q(t)=\sum_{j=1}^k\frac{P_L(\A_{k_j})(t)}{P(\A_{k_j})(t)}+\sum_{j=1}^M\frac{P_L(\D_{m_j})(t)}{P(\D_{m_j})(t)}$.
All roots of the first two factors of $P$ lie on the unit circle, so
one has to prove that all roots of $Q$ lie on the unit circle as
well.

Note that $Q(t)$ is a sum of ratios of palindromic polynomials and
one has $\deg P_L(\A_{k_j})-\deg P(\A_{k_j})=\deg P_L(\D_{m_j})-\deg
P(\D_{m_j})=-2$.

Taking the table from Proposition \ref{table} into account, we can
compute $\ph(x)=e^{2ix}Q(e^{2ix})$:
$$
\ph(x)=\sum_{j=1}^k\frac{\sin((2k_j-2)x)}{\sin(2k_jx)}+\sum_{j=1}^M\frac{\cos((m_j-4)x)}{\cos((m_j-2)x)}.
$$
Note that $\ph$ is an even $\pi$-periodic trigonometric function.

\begin{lemma}
Residues of poles of the functions
$f^\A_{k_j}=\frac{\sin((2k_j-2)x)}{\sin(2k_jx)}$ and
$f^\D_{m_j}=\frac{\cos((m_j-4)x)}{\cos((m_j-2)x)}$ are negative for
$x\in(0,\pi/2)$.
\end{lemma}

\begin{proof}
Consider first $f^\A_{k_j}$. Let $x$ be a pole of $f^\A_{k_j}$,
i.~e. $2k_jx=n\pi$ for some $n\in\Z$. Then the residue of
$f^\A_{k_j}$ at $x$ is equal to
$\frac{\sin((2k_j-1)x)}{2k_j\cos(2k_jx)}$.

If $n$ is odd, then $\cos(2k_jx)=-1$, and
$\sin((2k_j-2)x)=\sin(n\pi-2x)=\sin2x$, which is positive for
$x\in(0,\pi/2)$.

Similarly, if $n$ is even, then $\cos(2k_jx)=1$, and
$\sin((2k_j-2)x)=\sin(n\pi-2x)=-\sin2x$, which is negative for
$x\in(0,\pi/2)$.

For $f^\D_{m_j}$ the argument is similar. Poles of $f^\D_{m_j}$ are
at $x=(\pi/2+n\pi)/(m_j-2)$. The residue is equal to
$-\frac{\cos((m_j-4)x)}{(m_j-2)\sin((m_j-2)x)}$.

For even $n$ we have $\sin((m_j-2)x))=1$, and
$\cos((m_j-4)x)=\cos(\pi/2-2x)$, which is positive for
$x\in(0,\pi/2)$.

For odd $n$ we have $\sin((m_j-2)x))=-1$, and
$\cos((m_j-4)x)=\cos(-\pi/2-2x)$, which is negative for
$x\in(0,\pi/2)$.
\end{proof}

Therefore by Lemma \ref{lem:roots} the number of zeros of $\ph$ on
$[0,\pi/2]$ is at least the number of poles of $\ph$ on $[0,\pi/2]$
minus 1. Since $\ph$ is even, the same holds for the interval
$[-\pi/2,0]$. We conclude that the number of different zeros of
$\ph(x)$ on $(-\pi/2,\pi/2)$ is at least the number of its poles on
$(-\pi/2,\pi/2)$ minus 2 (note that $\ph(x)$ has no poles at $0$ and
$\pm\pi/2$).

Equivalently, the number of zeros of $Q$ on the unit circle is at
least the number of its poles on the unit circle minus 2.

The rational function $Q$ has a double zero at infinity. Therefore
the number of all finite zeros counted with multiplicities is equal
to $N-2$, where $N$ is  the number of poles of $Q$ counted with
multiplicities. But the poles of $Q$ are simple and lie on the unit
circle. Therefore the number of zeros of $Q$ on the unit circle is
already at least $N-2$, i.~e. the maximum possible. We conclude that
all zeros of $Q$ are on the unit circle (and they are all simple).

\

\noindent\emph{The case of $S\oplus\E_7^{\oplus l}$} : The
Poincar\'e polynomial $P_L(S\oplus\E_7^{\oplus l})$ in Theorem
\ref{main} has the following form:
$$
P=\left[\(P(\E_7)(t)\)^l\prod_{j=1}^kP(\A_{k_j})(t)\prod_{j=1}^M
P(\D_{m_j})(t)\right]Q(t),
$$
where $Q=l\frac{P_L(\E_7)(t)}{P(\E_7)(t)}+\sum_{j=1}^k
\frac{P_L(\A_{k_j})(t)}{P(\A_{k_j})(t)}+\sum_{j=1}^M
\frac{P_L(\D_{m_j})(t)}{P(\D_{m_j})(t)}$.

Taking Proposition \ref{table} into account, we see that the real
function $\ph(x)=Q(e^{2ix})$ has the form
$$
\ph(x)=l\frac{\sin4x\cos x}{\sin7x}+\sum_{j=1}^k
\frac{\sin((2k_j-2)x)}{\sin(2k_jx)}+\sum_{j=1}^M
\frac{\cos((m_j-4)x)}{\cos((m_j-2)x)}.
$$
Again, the function $\ph$ is even and $\pi$-periodic.

The function $f^{\E_7}=\frac{\sin4x\cos x}{\sin7x}$ has $3$ poles in
the interval $[0,\pi/2]$ at points $k\pi/7$, $k=1,2,3$. The residues
are negative at $\pi/7$, $2\pi/7$ and positive at $3\pi/7$.

Therefore by Lemma \ref{lem:roots} the number of zeros of $\ph$ on
the interval $(0,\pi/2)$  is at least the number of poles of $\ph$
on this interval minus 3.

Let us apply Lemma \ref{lem:roots2} to $\ph$ and $[0,\pi/2]$. First,
$\ph(0)=l\cdot4/7+\sum\frac{k_j-1}{k_j}+M>0$. Also,
$\ph(\pi/2)=-\sum\frac{k_j-1}{k_j}+\sum\lim_{x\to\pi/2}\frac{\cos((m_j-4)x)}{\cos((m_j-2)x)}<0$,
since $\lim_{x\to\pi/2}\frac{\cos((m_j-4)x)}{\cos((m_j-2)x)}$ is
either $-1$ or $-\frac{m-4}{m-2}$ depending on whether $m_j$ is even
or odd. Therefore the number $c$ in Lemma \ref{lem:roots2} is equal
to $1$.

We conclude that the number of zeros of $\ph$ on the interval
$(0,\pi/2)$ differs from the number of poles of $\ph$ on this
interval by an odd number.

Again, since $\ph$ is even, the number of different roots of $Q$ on
the unit circle is at least the number of its poles on the unit
circle minus twice an odd number.

As before, $Q$ has a double zero at infinity, and does not have
poles outside the unit circle. Therefore the total number of zeros
of $Q$ counted with multiplicities is equal to the number of poles
of $Q$ lying on the unit circle minus 2. Subtracting the zeros lying
on the unit circle we get that the number of zeros not lying on the
unit circle is either 4 or 0.

\subsection{Final notes}
Of course one would like to have more conceptual proof of the above
theorem, e.~g. by relating to a semisimple singularity some
naturally defined unitary operator whose characteristic polynomial
would coincide with the Poincar\'e polynomial of the corresponding
Lie algebra.

Results of some numerical computations reproduced in the table below
show that for the Poincar\'e polynomials of Lie algebras
corresponding to $S\oplus\E_7^{\oplus l}$ both possibilities
mentioned in Theorem 1 are indeed realised. That is, there might
occur 4 non-unimodular roots, and there might be none too.

\

\

\begin{center}
\begin{tabular}{r|ccc}
&\multicolumn3c{Number of roots outside the unit circle for}\\
$k$&\quad\ \ $\A_k\oplus\E_7$\ \ \ \ \ \ &$\D_{2k}\oplus\E_7$\ \ &$\D_{2k+1}\oplus\E_7$\ \ \ \ \ \ \\
\hline
2&&&0\\
3&&0&0\\
4&0&0&0\\
5&0&4&0\\
6&0&4&0\\
7&0&4&0\\
8&4&0&4\\
9&4&0&4\\
10&4&0&4\\
11&4&0&4\\
12&0&4&0\\
13&0&4&0\\
14&0&4&0\\
15&4&4&\\
16&4&&\\
\end{tabular}
\end{center}

\end{document}